\documentclass{amsart}

\usepackage{amssymb,color}

\newcommand{\f}{\ensuremath{\varphi}}
\newcommand{\p}{\ensuremath{\psi}}
\renewcommand{\a}{\ensuremath{\alpha}}
\renewcommand{\b}{\ensuremath{\beta}}
\newcommand{\si}{\sigma}
\newcommand{\De}{\Delta}
\newcommand{\Ga}{\Gamma}
\newcommand{\Si}{\Sigma}

\newcommand{\mdl}[1]{\models_{\lgc{#1}}}
\newcommand{\imp}{\Rightarrow}
\newcommand{\lang}{{\mathcal L}}
\newcommand{\alg}[1]{\mathbf{#1}}
\newcommand{\pos}[1]{\mathbf{#1}}
\newcommand{\A}{\alg{A}}
\newcommand{\B}{\alg{B}}
\newcommand{\E}{\alg{E}}
\newcommand{\cop}[1]{\mathbb{#1}}
\newcommand{\eq}{\approx}
\newcommand{\leqn}{\preccurlyeq}
\newcommand{\eleq}{\sqsubseteq}
\newcommand{\lgc}[1]{\mathrm{#1}}

\newcommand{\K}{{\ensuremath{\mathcal{K}}}}
\newcommand{\F}{{\ensuremath{\alg{F}}}}
\newcommand{\V}{{\ensuremath{\mathcal{V}}}}
\newcommand{\U}{{\ensuremath{\mathcal{U}}}}
\newcommand{\var}{{\ensuremath{\mathrm{Var}}}}
\newcommand{\FpV}{\Fp_{\V}}
\newcommand{\Fp}{\alg{Fp}}
\newcommand{\tp}{\mathrm{type}}
\newcommand{\cg}{\mathrm{Con}}
\newcommand{\ExUn}{\mathsf{E}}
\newcommand{\EUAlg}{\mathsf{C}}
\renewcommand{\U}[2]{\mathsf{U}_{#1}(#2)}

\newcommand{\FML}{\alg{Fm}_{\lang}(\omega)}
\newcommand{\FreeV}{\F_{\V}}
\newcommand{\Fml}{\alg{Fm}_{\lang}}

\newtheorem{theorem}{Theorem}

\newtheorem{lemma}[theorem]{Lemma}
\newtheorem{example}[theorem]{Example}
\newtheorem{corollary}[theorem]{Corollary}

\begin{document}

\title[Exact Unification]{Exact Unification}

\author[G. Metcalfe]{George Metcalfe}	
\address{Mathematical Institute, University of Bern,  Switzerland}	
\email{george.metcalfe@math.unibe.ch}  
\thanks{The first author acknowledges support from Swiss National Science Foundation grant 200021{\_}146748.}	
\author[L.M. Cabrer]{Leonardo M. Cabrer}	
\address{Department of Statistics, Computer Science and Applications, University of Florence, Italy}	
\thanks{The second author research was supported by a Marie Curie Intra European Fellowship within  the [European Community's] Seventh Framework Programme [FP7/2007-2013] under the Grant Agreement n. 326202.}	
\email{l.cabrer@disia.unifi.it}  



\begin{abstract}
A new hierarchy of ``exact" unification types is introduced, motivated by the study of admissibility 
for equational classes and non-classical logics. 
In this setting, unifiers of identities in an equational class are preordered, not by instantiation, 
but rather by inclusion over the corresponding sets of unified identities. Minimal complete sets of 
unifiers under this new preordering always have a smaller or equal cardinality than those provided by 
the standard instantiation preordering, and in significant cases a dramatic reduction 
may be observed. In particular, the classes of distributive lattices, idempotent semigroups, and MV-algebras, 
which all have nullary unification type, have unitary or finitary exact type. 
These results are obtained via an algebraic interpretation of exact unification, inspired by Ghilardi's 
algebraic approach to equational unification. 
\end{abstract}


\maketitle

 
\section{Introduction}

It has long been recognized that the study of admissible rules is inextricably bound up with the study of equational 
unification (see, e.g.,~\cite{Ryb97,Ghi99,Ghi00}). Indeed, from an algebraic perspective, admissibility in an equational 
class (variety) of algebras may be viewed as a generalization of unifiability in that class, and 
conversely,  checking admissibility may be reduced to comparing certain sets of unifiers.
%
%
This paper provide a new classification of equational unification problems 
that simplifies such reductions.\footnote{The reader is referred to~\cite{BS81} and~\cite{ML71} for basic notions 
of universal algebra and category theory, respectively.} 

Let us fix an equational class of algebras $\V$ for a language $\lang$ 
and denote by $\Fml(X)$, the {\em formula algebra} 
(absolutely free algebra or term algebra) of $\lang$ over a set of variables $X \subseteq \omega$. 
 A substitution (homomorphism) $\si \colon \Fml(X) \to\FML$ 
 is called a {\em $\V$-unifier (over $X$)} of  a set of $\lang$-identities $\Si$ with variables in $X$ if  
\[
\V \models \si(\f) \eq \si(\p) \quad \mbox{for all $\f \eq \p$ in $\Si$.}
\] 
A clause $\Si \imp\De$  (an ordered pair of finite sets of $\lang$-identities $\Si,\De$) 
 is  {\em $\V$-admissible} if for each substitution $\si \colon \Fml(X) \to\FML$ 
 where the variables in $\Si \cup \De$ are contained in $X$,
 \[
 \mbox{$\si$ is a $\V$-unifier of $\Si$} 
 \quad \Rightarrow \quad
  \mbox{$\si$ is a $\V$-unifier of some member of $\De$.} 
 \]
 In particular, $\Si$ is $\V$-unifiable if and only if $\Si \imp \emptyset$ is not $\V$-admissible.

Now suppose that the unification type of $\V$ is at most finitary, meaning that every $\V$-unifier 
of a set of $\lang$-identities $\Si$ over the variables in $\Si$  is a substitution instance of one 
of a finite set $S$ of $\lang$-unifiers of $\Si$. Then a clause $\Si \imp \De$ is $\V$-admissible if each member  
of $S$ is an $\lang$-unifier of a member of $\De$. If there is an algorithm for determining the finite basis set $S$ 
for $\Si$ and the equational theory of $\V$ is decidable, then checking $\V$-admissibility is also decidable. 
This observation, together with the pioneering work of Ghilardi on equational unification for classes of Heyting 
and modal algebras~\cite{Ghi99,Ghi00}, has led to a wealth of decidability, complexity, and axiomatization results for 
admissibility in these classes and corresponding modal and intermediate 
logics~\cite{Iem01,Iem05,Jer05,CM10,BR11a,BR11b,OR13,Jer13b}.

The success of this approach to admissibility appears to rely on considering varieties with at most finitary unification type. 
This is not a necessary condition, however,  as illustrated by the case of MV-algebras, the algebraic semantics of \L ukasiewicz 
infinite-valued logic. Decidability, complexity, and axiomatization results for admissibility in this class have been established by 
Je{\v r}{\'a}bek~\cite{Jer09a,Jer09b,Jer13b} 
via a similar reduction of finite sets of identities to finite approximating sets of identities.  
On the other hand, it has been shown by Marra and Spada~\cite{MS13} that 
the class of MV-algebras has nullary unification type. This means in particular that there are finite sets of identities 
for which no finite basis of unifiers exists. Further examples of this discrepancy may be found in~\cite{CM13}, including 
 the very simple example of the class of distributive lattices where admissibility and validity of clauses coincide 
but unification is nullary.

As mentioned above, it is possible to check the $\V$-admissibility of a clause $\Si \imp \De$ by checking 
that every $\V$-unifier of $\Si$ in a certain ``basis set'' $\V$-unifies $\De$. Such a basis set $S$ typically has 
the property that every other $\V$-unifier of $\Si$ is a substitution instance of a member of $S$.
The starting point for this paper is the observation that a weaker condition on $S$ suffices, 
leading potentially to smaller sets of $\V$-unifiers. What is really required for checking admissibility is the property that 
each $\V$-unifier of $\Si$ is also a $\V$-unifier of all identities $\V$-unified by some particular member of $S$. 
Then  $\Si \imp \De$ is $\V$-admissible if each member of $S$ is a $\V$-unifier of a member of $\De$. 
This leads to a new ordering of $\V$-unifiers and hierarchy of exact (unification) types. 

We obtain also a Ghilardi-style algebraic characterization of exact unification, where the role of formulas is taken 
by the finitely presented algebras of the equational class. In Ghilardi's approach, a unifier is a homomorphism 
from a finitely presented algebra into a projective algebra of the class, and unifiers are preordered by composition of  
homomorphisms. Here, coexact unifiers are defined as homomorphisms from a finitely presented algebra onto 
an exact algebra (an algebra that embeds into the free algebra of the class on countably infinitely many generators) 
and the preordering remains the same. This contrasts with the syntactic approach to exact unification where the 
unifiers remain unchanged but a new preorder is introduced. Nevertheless, the syntactic and algebraic exact unification types 
coincide as in the standard approach.

Although certain equational classes have the same exact type as unification type (in particular, any equational class of 
unitary type), crucially we obtain examples where the exact type is smaller.  
In particular, distributive lattices have unitary exact type, while idempotent semigroups, various classes of pseudo-complemented 
distributive lattices, and MV-algebras have finitary exact type. 
We also provide an example (due to R.~Willard) of an equational class of infinitary unification type but finitary exact type.

We proceed as follows. In Section~\ref{s:equational_unification_and_admissibility}, we recall standard notions of equational 
unification and admissible rules, and Ghilardi's algebraic account of unification types. In Section~\ref{s:exact}, we introduce 
the new notion of exact unifier and exact unification types, providing an algebraic interpretation and applications in 
Section~\ref{s:algebraic}. Several cases studies are considered in Section~\ref{s:case_studies} 
and some ideas for further research are presented in Section~\ref{s:concluding_remarks}.


\section{Equational Unification and Admissibility}\label{s:equational_unification_and_admissibility}

In this section, we describe briefly some key ideas from the theory of 
equational unification (referring to~\cite{BS01} for further details) 
and their relevance to the study of admissible rules. We recall, in particular,  
the unification type of a finite set of identities in an equational class and 
the algebraic interpretation of unification types provided by Ghilardi in~\cite{Ghi97}. 
These ideas, and also developments in subsequent sections, are most elegantly 
presented in the  general setting of preordered sets.

Let $\pos{P} = \langle P, \le \rangle$ be a preordered set (i.e., $\le$ is a reflexive and transitive binary relation on $P$). 
A {\it complete} set for $\pos{P}$ is a subset $M \subseteq P$ such that for every $x \in P$, there exists $y \in M$ satisfying 
$x \le y$. A complete set $M$ for $\pos{P}$ is called a {\it $\mu$-set} for $\pos{P}$ if  $x \not \le y$ and $y \not \le x$ 
for all distinct $x,y \in M$. It is easily seen that if $\pos{P}$ has a $\mu$-set, then every $\mu$-set of $\pos{P}$ has the 
same cardinality. Hence $\pos{P}$ may be said to be {\em nullary} if it  has no $\mu$-sets ($\tp(\pos{P}) = 0$), {\em infinitary} if it has 
a $\mu$-set of infinite cardinality ($\tp(\pos{P}) = \infty$), {\em finitary} if it has a finite $\mu$-set of cardinality greater than $1$ 
($\tp(\pos{P}) = \omega$), and {\em unitary} if it has a $\mu$-set of cardinality 1 ($\tp(\pos{P}) = 1$). These types are ordered as 
follows: $1<\omega<\infty<0$. 

The following useful lemma demonstrates that the type of a preordered set may be viewed as a 
categorical invariant.

\begin{lemma}\label{Lemma:EquivPreorder}
Suppose that two preordered sets $\langle P, \le \rangle$ and  $\langle Q, \le \rangle$ 
are equivalent: i.e., there exists a map $e \colon P \to Q$ such that\vspace{.05in}
\begin{enumerate}
\item[(1)]	for each $q \in Q$, there is a $p \in P$ such that $e(p) \le q$ and $q \le e(p)$\vspace{.05in}
\item[(2)]	for each $p_1,p_2 \in P$, $p_1 \le p_2$ iff $e(p_1) \le e(p_2)$.
\end{enumerate}
Then $\langle P, \le \rangle$ and  $\langle Q, \le \rangle$ have the same type.
\end{lemma}

We turn our attention now to the syntactic account of equational unification. 
Let us fix $\lang$ to be an algebraic language and $\V$ an equational class of $\lang$-algebras 
(equivalently, a variety: a class of $\lang$-algebras closed under taking products, subalgebras, and homormophic 
images).\footnote{The results of this paper 
are also valid for quasi-equational classes and, more generally, 
for prevarieties (classes of algebras closed under products, 
subalgebras and isomorphic images). However, as all of our examples and the vast majority of cases considered in 
the literature are equational classes, we restrict our account to this slightly simpler setting.} 
Let $X \subseteq \omega$ be a set of variables,
 and  consider substitutions 
$\si_i \colon \Fml(X) \to\FML$ for $i = 1,2$. We say that $\si_1$ is {\em more general} than $\si_2$  (written $\si_2 \leqn \si_1$) if there exists a substitution $\si' \colon \FML\to \FML$ such that $\si' \circ \si_1 = \si_2$. 

Let $\Si$ be a finite set of $\lang$-identities, denoting the variables occurring in $\Si$ by $\var(\Si)$. Then 
$\U{\V}{\Si}$ is defined as the set of $\V$-unifiers of $\Si$ over $\var(\Si)$ preordered by $\leqn$. 
For $\U{\V}{\Si} \not = \emptyset$,  the {\em $\V$-unification type} of $\Si$ is defined as  $\tp(\U{\V}{\Si})$.  
The {\em unification type of $\V$} is the maximal type of a $\V$-unifiable finite set $\Si$ of $\lang$-identities.

\begin{example}\label{e:equnif}
Equational unification has been studied for a wide range of equational classes. 
In the most general setting of {\em syntactic unification} where $\V$ is the class of all $\lang$-algebras, 
every syntactically unifiable finite set $\Si$ of $\lang$-identities has  a most general unifier; 
that is, syntactic unification is unitary (see, e.g.,~\cite{BS01}). The class of {\em Boolean algebras} 
is also unitary~\cite{BS87}: if $\{\f \eq \top\}$ is unifiable (equivalent to the satisfiability of $\f$), 
then it has a most general unifier defined by $\si(x) = \lnot \f \lor x$ 
for each $x \in \var(\f)$. The class of {\em Heyting algebras} is not unitary; for example, 
$\{x \lor y \eq \top\}$ has a $\mu$-set of unifiers $\{\si_1,\si_2\}$ where $\si_1(x) = \top$, $\si_1(y) = y$, 
$\si_2(x) = x$, $\si_2(y) = \top$. It is, however, finitary~\cite{Ghi99}. More problematically, 
the class of {\em semigroups} is infinitary~\cite{Plo72}: e.g., $\{x \cdot y \eq y \cdot x\}$ has  a $\mu$-set 
$\{\si_{m,n} \mid \gcd(m,n)= 1\}$ where $\si_{m,n}(x) = z^m$ and $\si_{m,n}(y) = z^n$. 
Many familiar classes of algebras are nullary; in particular, 
the class of {\em distributive lattices} has nullary unification type (see~\cite{Ghi99}); 
e.g., $\{x \land y \eq z \lor w\}$ has no $\mu$-set. Other nullary classes of algebras 
include {\em idempotent semigroups (bands)}~\cite{Ba86}, {\em pseudo-completemented 
distributive lattices}~\cite{Ghi97}, and {\em MV-algebras}~\cite{MS13}.
\end{example}

We now recall Ghilardi's algebraic account of equational unification~\cite{Ghi97}. 
Let $\FreeV(X)$ denote the free $\lang$-algebra of $\V$ over a set of 
variables $X$ and let $h_\V \colon \Fml(X) \to \FreeV(X)$ be the canonical homomorphism (that is, the unique homomorphism that acts as the identity on the elements of $X$). 
Given a finite set of $\lang$-identities $\Si$ and a finite set $X \supseteq \var(\Si)$, we denote by $\FpV(\Si,X)$ 
the algebra in $\V$ finitely presented by $\Si$ and $X$: that is, the quotient algebra $\FreeV(X) / \Theta_\Si$ where 
$\Theta_\Si$ is the congruence on $\FreeV(X)$ generated by the set $\{(h_\V(\f),h_\V(\p)) \mid \f \eq \p \in \Si\}$. We also let
$\mathsf{FP}(\V)$ denote the class of finitely presented algebras of $\V$.

Given $\A \in \mathsf{FP}(\V)$, a homomorphism $u \colon \A \to \B$ is called a {\em unifier} for $\A$ if 
 $\B \in \mathsf{FP}(\V)$ is {\em projective} in $\V$: that is, there exist homomorphisms $\iota \colon \B \to \FreeV(\omega)$ and 
 $\rho \colon \FreeV(\omega) \to \B$ such that $\rho \circ \iota$ is the identity map on $B$. Let $u_i \colon \A \to \B_i$ for $i=1,2$ be unifiers 
for $\A$. Then $u_1$ is {\em more general than} $u_2$, written $u_2 \le u_1$, if there exists a homomorphism 
$f \colon \B_1 \to \B_2$ such that $f \circ u_1 = u_2$. 

Let $\U{\V}{\A}$ be the set of unifiers of $\A \in \mathsf{FP}(\V)$ preordered by $\le$.  
For $\U{\V}{\A}\neq\emptyset$, the {\em unification type of $\A$ in $\V$} is 
defined as $\tp(\U{\V}{\A})$ and the {\em algebraic unification type} of $\V$ is the maximal type of 
$\A$ in $\mathsf{FP}(\V)$ such that $\U{\V}{\A}\neq\emptyset$. 

\begin{theorem}[Ghilardi~\cite{Ghi97}]
Let $\Si$ be a $\V$-unifiable finite set of identities and let  $\A$ denote the finitely presented algebra $\FpV(\Si,\var(\Si))$. Then
\[
\tp\bigl(\U{\V}{\Si}\bigr) = \tp\bigl(\U{\V}{\A}\bigr).
\]
Hence the algebraic unification type of $\V$ coincides with the unification type of $\V$.
\end{theorem}

Let us see now how these ideas relate to the notion of admissibility defined in the introduction. 
Recall that the {\em kernel} of a homomorphism $h \colon \A \to \B$  is defined as 
\[
\ker(h) = \{(a,b) \in A^2 \mid h(a)=h(b)\}.
\]
In what follows, we will freely identify $\lang$-identities with pairs of $\lang$-formulas. We will also say that a 
$\lang$-clause $\Si \imp \De$ is {\em valid} in a class of $\lang$-algebras $\K$, written $\K \models \Si \imp \De$, 
if the universal sentence $(\forall \bar{x})(\bigwedge \Si \imp \bigvee \De)$ is valid in each algebra in $\K$.

\begin{lemma}\label{Lemma:Crucial}
  Let $\Si\cup\De$ be a finite set of $\lang$-identities. Then the following  are equivalent:
\begin{itemize}
\item[(i)]	$\Si\imp \De$ is admissible in $\V$.
\item[(ii)]	$\FreeV(\omega)\models \Si\imp \De$. 
\item[(iii)] 	For each $\si\colon \FML\to \FML$ such that $\Si\subseteq \ker(h_{\V}\circ\si)$,\[\De\cap\ker(h_{\V}\circ\si)\neq\emptyset.\]
\end{itemize}
If in particular $\De=\{ \f\eq\p\}$, then (i)-(iii) above are also equivalent to
\begin{itemize}
\item[(iv)] $(\f,\p)\in\bigcap \{\ker(h_{\V}\circ \si)\mid \si\colon \FML\to \FML\mbox{ and }\Si\subseteq \ker(h_{\V}\circ\si)\}$.
\end{itemize}
\end{lemma}

\begin{proof}
(i)$\Rightarrow$(ii) 
Suppose that $\Si \imp \De$ is admissible in $\V$ and let $g \colon \FML \rightarrow \FreeV(\omega)$ be a 
homomorphism such that $\Si \subseteq \ker g$. Let $\si$ be a map 
sending each variable $x$ to a member of the equivalence class $g(x)$. By the universal mapping property for 
$\FML$, this extends to a homomorphism $\si \colon \FML \rightarrow \FML$. 
But $h_\V (\si(x)) = g(x)$ for each variable $x$, so $h_\V \circ \si = g$. 
Hence, for each  $\f' \eq \p' \in \Si$, also $h_\V(\si(\f')) =  h_\V(\si(\p'))$, i.e.,  $\V \models\si(\f') \eq \si(\p')$. 
Therefore, $\si$ is a unifier of $\Si$ and, by assumption, $\V \models \si(\f) \eq \si(\p)$ for some $\f\eq \p\in\De$. 
It follows that  
$g(\f) = h_\V(\si(\f)) = h_\V(\si(\p)) = g(\p)$ as required.

(ii)$\Rightarrow$(iii)
 Let $\si\colon \FML\rightarrow\FML$ be such that $\Si\subseteq \ker(h_{\V}\circ\si)$, that is, $\V \models \si(\Si)$. Therefore $\FreeV(\omega)\models \si(\Si)$.
 By assumption, there exists and equation $\f\eq\p\in\De$ such that  $\FreeV(\omega)\models \si(\f)\eq\si(\p)$, that is, $\V \models \si(\f)\eq\si(\p)$. Hence, $(\f,\p)\in\ker(h_{\V}\circ\si')\cap \De$.

(iii)$\Rightarrow$(i)  
Let $\si\colon \FML\rightarrow \FML$ be such that $\V \models\si(\Si)$, 
that is, $\Si\subseteq \ker(h_{\V}\circ\si)$. 
By hypothesis
there exists $\f\eq\p\in\De\cap\ker(h_{\V}\circ\si)$.
Then $h_{\V}(\si(\f))=h_{\V}(\si(\p))$, i.e., $\V \models\si(\f)\eq\si(\p)$. We obtained that $\Si\imp \De$ is admissible in $\V$.

If $\De=\{\f\eq\p\}$, (iii) is equivalent to (iv).
\end{proof}

Suppose now that $\V$ is any equational class of $\lang$-algebras and that $S$ is a $\mu$-set 
for the $\leqn$-preordered set of $\V$-unifiers of a finite set of $\lang$-identities $\Ga$. Then clearly:
\[
\Ga \imp \De \mbox{ is $\V$-admissible}  \quad \Leftrightarrow \quad 
\mbox{each $\si \in S$ is a $\V$-unifier of some $(\f \eq \p) \in \De$.}
\]
Note in particular that if $\V$ is unitary or finitary and there exists an algorithm for finding $\mu$-sets, 
then checking admissibility in $\V$ is decidable whenever the equational theory of $\V$ is decidable. 
There are, however, many well-known equational classes having infinitary or nullary unification type, for which 
such a method is unavailable. The starting point for the new approach described below is the observation 
that the above equivalence can hold even when $S$ is not a $\mu$-set for the $\leqn$-preordered set of $\V$-unifiers. 
More precisely, it is enough that each $\si \in \U{\V}{\Ga}$\ \,$\V$-unifies all identities $\V$-unified by some particular 
member of $S$.


\section{Exact Unifiers}\label{s:exact}

We begin by defining a new preorder on substitutions relative to a fixed equational class of $\lang$-algebras $\V$. 
Let  $X \subseteq \omega$ be a set of variables and let $\si_i \colon \Fml(X)\to\FML$ 
be substitutions for $i=1,2$.  We write $\si_2\eleq_{\V} \si_1$ if all identities $\V$-unified by $\si_1$ are $\V$-unified by $\si_2$. 
 More precisely:
\[
\si_2\eleq_{\V} \si_1 \quad \Leftrightarrow \quad  \ker(h_{\V}\circ\si_1)\subseteq\ker(h_{\V}\circ\si_2).
\]
Clearly, $\eleq$ is a preorder on substitutions of the form $\si \colon \Fml(X) \to\FML$. Moreover:

\begin{lemma}\label{l:easier}
Given $X \subseteq \omega$ and substitutions $\si_i \colon \Fml(X)\to\FML$  for $i=1,2$:
\[
\si_2\leqn\si_1 \quad \Rightarrow \quad \si_2\eleq_{\V} \si_1.
\]
\end{lemma}
\begin{proof}
Suppose that $\si_2\leqn\si_1$. Then there exists a substitution  $\si' \colon \FML\to\FML$ such that $\si' \circ \si_1 = \si_2$. 
But then if $h_\V \circ \si_1(\f) = h_\V \circ \si_1(\p)$, also $h_\V \circ \si' \circ \si_1(\f) = h_\V \circ \si' \circ \si_1(\p)$ 
That is, $h_\V \circ \si_2(\f) = h_\V  \circ \si_2(\p)$. 
\end{proof}

Given  a  finite set $\Si$ of $\lang$-identities and  $X\supseteq \var(\Si)$, $\ExUn_{\V}(\Si,X)$ is defined as 
the set of $\V$-unifiers of $\Si$ over $X$ preordered by $\eleq_{\V}$. 
For $X=\var(\Si)$, we simply write $\ExUn_{\V}(\Si)$ instead of $\ExUn_{\V}(\Si,X)$. 
Let us also define for $Y \subseteq X$ and a substitution $\si\colon\Fml(Y)\to \FML$, 
the unique extension $\si_X \colon \Fml(X) \to \FML$ of $\si$ as
\[
\si_X(x) =
\begin{cases}
\si(x)&\mbox{if }x\in Y;\\
x&\mbox{otherwise}.
\end{cases}
\]

\begin{lemma}\label{Lem:Independence}
Let $\Si$ be a finite set of identities and $X\supseteq\var(\Si)$. 
Then
\[
	\tp(\ExUn_{\V}(\Si,X)) = \tp(\ExUn_{\V}(\Si))
\]
\end{lemma}
\begin{proof}
Let $Y=\var(\Si)$ and 
$(\_){\upharpoonright}_Y\colon\ExUn_{\V}(\Si,X)\to\ExUn_{\V}(\Si,Y)$ 
be the map that assigns each unifier of $\Si$ on $X$ to its restriction 
to the variables in $Y$.
It is easy to see that $(\_){\upharpoonright}_Y$ preserves 
$\eleq_{\V}$.
Let $(\_)_X\colon\ExUn_{\V}(\Si,Y)\to\ExUn_{\V}(\Si,X)$  be the map
defind by $\si\to\si_X$. 
It is clear that $(\_)_X$ preserves $\eleq_{\V}$ and that ${\si_X}{\upharpoonright}_Y=\si$ for each $\si\in \ExUn_{\V}(\Si,Y)$. This proves that  $\tp(\ExUn_{\V}(\Si,Y)) = \tp(\ExUn_{\V}(\Si))\leq \tp(\ExUn_{\V}(\Si,X))$. 

To see that  $ \tp(\ExUn_{\V}(\Si))\geq \tp(\ExUn_{\V}(\Si,X))$, let $\si\in \ExUn_{\V}(\Si,X)$. Assume without loss of generality that 
$\var(\si(x))\cap X=\emptyset$ for each $x\in X$. Define $\lambda\colon\Fml(X)\to \FML$ by 
\[
\lambda(x)=\begin{cases}
x&\mbox{if }x\in Y;\\
\si(x)&\mbox{otherwise}.
\end{cases}
\]
Then $\si=\lambda\circ (\si{\upharpoonright}_Y)_X$, i.e., $\si\leqn (\si{\upharpoonright}_Y)_X$. Hence, if $S\subseteq\ExUn_{\V}(\Si,Y)$ is a complete set, $\{\gamma_X\mid \gamma\in S\}$ is a complete set for $\ExUn_{\V}(\Si,X)$. Thus $ \tp(\ExUn_{\V}(\Si))\geq \tp(\ExUn_{\V}(\Si,X))$.
\end{proof}

Suppose that $\Si$ is a finite set of $\lang$-identities and $\ExUn_{\V}(\Si) \neq \emptyset$. Then 
 the {\it exact type of $\Si$ in $\V$} is defined as $\tp(\ExUn_{\V}(\Si))$. 
We also define the {\em exact unification type of $\V$} to be
 the maximal exact type of a $\V$-unifiable finite set $\Si$ of $\lang$-identities.
 
 Note that, because $\si_2\leqn\si_1$ 
implies $\si_2\eleq_{\V} \si_1$ (Lemma~\ref{l:easier}), every complete set for $\U{\V}{\Si}$ 
is also a complete set for  $\ExUn_{\V}(\Si)$.  Hence, for 
 $\tp(\U{\V}{\Si}) \in \{1,\omega\}$:
\[
\tp(\ExUn_{\V}(\Si)) \le \tp(\U{\V}{\Si}).
\] 
Moreover, using Lemmas~\ref{Lemma:Crucial} and~\ref{Lem:Independence} we obtain:

\begin{corollary}\label{Cor:AdmissibleExact}
Let $\Si\cup\De$ be a finite set of $\lang$-identities and $S$ a complete set 
for $\ExUn_{\V}(\Si)$.  Then the following statements are equivalent:
\begin{itemize}
\item[(i)]	$\Si\imp \De$ is admissible in $\V$.
\item[(ii)] 	For each $\si \in S$, the unifier $\si_X\colon\Fml(\var(\Si\cup\De))\to \FML$ is a $\V$-unifier of some $\f \eq \p \in \De$. 
\item[(iii)] 	For each $\si\in S$, $\De\cap\ker(h_{\V}\circ\si_X)\neq\emptyset$.
\end{itemize} 
\end{corollary}

The close connection between exact types and admissible rules is also witnessed by the following result.

\begin{theorem}\label{Theo:TypeAndSizeAdm}
If an $\lang$-clause $\Si\imp\De$ is $\V$-admissible and 
 $\ExUn_{\V}(\Si)$ has a finite $\mu$-set $S$, then there exists 
$\De'\subseteq \De$ such that 
$|\De'|\leq |S|$ and $\Si\imp\De'$ 
is $\V$-admissible.
\end{theorem}
\begin{proof}
Let $X=\var(\Si\cup\De)$ and let $\{\si_1,\ldots,\si_n\}$ be a $\mu$-set for $\ExUn_{\V}(\Si,X)$. 
By Lemma~\ref{Lem:Independence}, $\tp(\ExUn_{\V}(\Si))=\tp(\ExUn_{\V}(\Si,X))$.
By Lemma~\ref{Lemma:Crucial}, for each $i\in\{1,\ldots,n\}$, there exists an identity 
$\f_i\eq\p_i \in \De$ such that $(\f_i,\p_i)\in\ker(h_{\V}\circ \si_i)$. Let $\De'=\{\f_1\eq\p_1, \ldots, \f_n \eq \p_n\}$. 
We claim that $\Si\imp\De'$ is admissible in $\V$. 
Suppose that $\si\colon \Fml(X)\to \FML$ satisfies $\Si\subseteq\ker(h_{\V}\circ\si)$. 
Since $\si\in\ExUn_{\V}(\Si,X)$ and $\{\si_1,\ldots,\si_n\}$ is  a $\mu$-set for $\ExUn_{\V}(\Si,X)$, 
there exists $i\in\{1,\ldots,n\}$ such that $\ker(h_{\V}\circ\si_i)\subseteq\ker(h_{\V}\circ\si)$. 
Hence  $(\f_i,\p_i)\in\ker(h_{\V}\circ\si)$, and the result follows. 
\end{proof}

A finite set $\Si$ of $\lang$-identities is said to be {\em admissibly reducible in $\V$} if whenever $\Si\imp\De$ is 
admissible in $\V$ for some non-empty set of $\lang$-identities $\De$, then there exists $\f\eq\p\in\De$ such that 
$\Si\imp\f\eq\p$ is admissible in $\V$.

\begin{corollary}
Let $\Si$ be a finite set of $\lang$-identities. If 
$\tp(\ExUn_{\V}(\Si))=1$, then $\Si$ is admissibly reducible in $\V$. 
Conversely, if $\tp(\ExUn_{\V}(\Si))\in\{1,\omega\}$ and $\Si$ is admissibly reducible in $\V$ then $\tp(\ExUn_{\V}(\Si))=1$.
\end{corollary}
\begin{proof}
The first claim follows immediately from the previous theorem. For the second claim, 
assume that $\tp(\ExUn_{\V}(\Si))\in\{\omega,1\}$ and that $\Si$ is admissibly reducible in $\V$. 
Then there exists a $\mu$-set $\{\si_1,\ldots,\si_n\}$ for $\ExUn_{\V}(\Si)$. 
For each $i,j\in\{1,\ldots, n\}$ such that $i\neq j$, consider $(\f_{ij},\p_{ij})\in\ker(h_{\V}\circ\si_i)\setminus\ker(h_{\V}\circ\si_j)$. Let $\De=\{\f_{ij}\eq\p_{ij}\mid i,j\in\{1,\ldots, n\} \mbox{ and }i\neq j\}$. 

Suppose that $n\neq 1$ and hence $\De\neq\emptyset$.
Since $\{\si_1,\ldots,\si_n\} $ is a $\mu$-set for $\ExUn_{\V}(\Si)$, by Corollary~\ref{Cor:AdmissibleExact}, it follows that 
$\Si\imp\De$ is admissible in $\V$. But, by assumption, there exists $\f_{ij}\eq\p_{ij}\in\De$ such that  $\Si\imp\f_{ij}\eq\p_{ij}$ 
is admissible in $\V$, contradicting the fact that $\V\not\models\si_j(\f_{ij})\eq\si_j(\p_{ij})$. 
We conclude that $n=1$, and hence that $\tp(\ExUn_{\V}(\Si))=1$.
\end{proof}


\section{Algebraic Co-Exact Unifiers} \label{s:algebraic}

We turn our attention now to the algebraic interpretation of exact unification. 
Following~\cite{DeJongh82}, a finite set  of $\lang$-identities $\Si$ will be called {\em exact in $\V$} if there exists 
a substitution $\si\colon\FML\rightarrow\FML$ such that for all $\a,\b\in\Fml(\var(\Si))$, 
\[
\V\models \Si\imp\{\a\eq\b \}
\quad \Leftrightarrow \quad 
\V\models \si(\a)\eq \si(\b).
\]
Note that by definition every exact set of identities is $\V$-unifiable.

Given a finite set of $\lang$-identities $\Si$ and a finite set of variables $X \supseteq \var(\Si)$, 
let $\rho_{(\Si,X,\V)}\colon\FreeV(X)\rightarrow \FpV(\Si,X)$ be the canonical quotient homomorphism 
from the free algebra $\FreeV(X)$ to the finitely presented algebra $\FpV(\Si,X)$.

\begin{lemma}\label{Lem:ExactAsSubAlg}
A finite set $\Si$ of $\lang$-identities is exact in $\V$ if and only if 
\[
\FpV\bigl(\Si,\var(\Si)\bigr)\in\cop{IS}\bigl(\FreeV(\omega)\bigr).
\]
\end{lemma}
\begin{proof}
($\Rightarrow$) Let $X=\var(\Si)$ and let  $\si\colon\FML\rightarrow\FML$ be a substitution 
such that for all $\a,\b\in\Fml(X)$, $\V\models \Si\imp\{\a\eq\b\}$
iff $\V\models \si(\a)\eq \si(\b)$. That is $\V\models \Si\imp\{\a\eq\b\}$ iff $h_{\V}(\si(\a))=h_{\V}(\si(\b))$. 
There is a unique homomorphism $\si'\colon \FreeV(\omega)\rightarrow \FreeV(\omega)$ such that $h_{\V}\circ\si=\si'\circ h_{\V}$ 
and hence  $h_{\V}(\Si)\subseteq \ker(\si')$.

 Let $\iota\colon\FreeV(X)\rightarrow \FreeV(\omega)$ be the inclusion map.
 Since $h_{\V}(\Fml(X))=\FreeV(X)$, it follows that $h_{\V}(\Si)\subseteq \ker(\si'\circ\iota)=\ker(\si')\cap \FreeV(X)^2$. 
There exists a unique $s\colon \FpV(\Si,X)\rightarrow\FreeV(\omega)$ such that $s\circ\rho_{(\Si,X,\V)}=\si '\circ\iota$. 
Let $a,b\in \FpV(\Si,X)$ be such that $s(a)=s(b)$ and $\a,\b\in\Fml(X)$ such that  $\rho_{(\Si,X,\V)}(h_{\V}(\a))=a$ and $\rho_{(\Si,X,\V)}(h_{\V}(\b))=b$. Then 
$$h_{\V}\circ\si(\a)=\si'\circ h_{\V}(\a)=(s\circ\rho_{(\Si,X,\V)}\circ h_{\V})(\a)=s(a)=s(b)$$
$$(s\circ\rho_{(\Si,X,\V)}\circ h_{\V})(\b)=\si'\circ h_{\V}(\b)=h_{\V}\circ\si_X(\b).$$
By assumption, $\Si\mdl{\V}\a\eq\b$. So $a=\rho_{(\Si,X,\V)}(h_{\V}(\a))=\rho_{(\Si,X,\V)}(h_{\V}(\b))=b$; i.e., 
$s$ is a one-to-one homomorphism. Hence, $\FpV(\Si,X)\in\cop{IS}(\FreeV(\omega))$.

($\Leftarrow$) Let $X=\var(\Si)$, and let 
$s\colon \FpV(\Si,X)\rightarrow\FreeV(\omega)$ be a  one-to-one homomorphism. 
Let $\si\colon\FML\rightarrow\FML$ be the unique homomorphism determined by its value on the variables as follows:
\[
\si(x)=\begin{cases}
\a_x& \mbox{ if }x\in X,\\
x&\mbox{ otherwise},
\end{cases}
\]
 where $\a_x$ is any formula 
such that $s(\rho_{(\Si,X,\V)}(h_{\V}(x)))=h_{\V}(\a_x)$. 
By induction on formula complexity, $h_{\V}(\si(\f))=s(\rho_{(\Si,X,\V)}(h_{\V}(\f)))$ for each 
$\f\in\Fml(X)$. Thus, if $\a,\b\in\Fml(X)$ are such that $h_{\V}(\si(\a))=h_{\V}(\si(\b))$, 
then $s(\rho_{(\Si,X,\V)}(h_{\V}(\a)))=s(\rho_{(\Si,X,\V)}(h_{\V}(\b)))$. Finally from the injectivity of  $s$ it follows that $\rho_{(\Si,X,\V)}(h_{\V}(\a))=\rho_{(\Si,X,\V)}(h_{\V}(\b))$, equivalently,
$\V\models \Si\imp\{\a\eq\b\}$. 
\end{proof}

We call an algebra~$\E$ {\em exact} in $\V$ if it is isomorphic to a finitely generated subalgebra of 
$\FreeV(\omega)$. By Lemma~\ref{Lem:ExactAsSubAlg} (see also~\cite{DeJongh82}), a finite  set of identities 
$\Si$ is exact iff the finitely presented algebra $\FpV(\Si,\var(\Si))$ is exact.

Given $\A \in \mathsf{FP}(\V)$, an onto homomorphism 
$u \colon \A \to \E$ is called a {\em coexact unifier} for~$\A$ if  $\E$ is exact.  
 Coexact unifiers are ordered in the same way as algebraic unifiers, that is, if $u_i \colon \A \to \E_i$ for $i=1,2$ are coexact unifiers for 
 $\A$, then $u_1 \le u_2$, if there exists a homomorphism $f \colon \E_1 \to \E_2$ such that $f \circ u_1 = u_2$. 

 Let $\EUAlg_{\V}(\A)$ be the  set of coexact unifiers for $\A$ preordered by $\leq$.  If $\EUAlg_{\V}(\A)\neq\emptyset$,  then the 
 {\em exact type} of $\A$ is defined as the type of $\EUAlg_{\V}(\A)$. The {\em exact algebraic unification type} 
 of $\V$ is the maximal exact type of $\A$ in $\V$ such that $\EUAlg_{\V}(\A) \neq\emptyset$.

 We obtain the following Ghilardi-style result. 

\begin{theorem}\label{Theo:EqualTypes}
Let $\V$ be an equational class and $\Si$ a finite set of $\V$-unifiable $\lang$-identities. Then for any $X\supseteq \var(\Si)$,
\[
\tp\bigl(\ExUn_{\V}(\Si)\bigr)=\tp\bigl(\ExUn_{\V}(\Si,X)\bigr)=\tp\bigl(\EUAlg_{\V}(\FpV(\Si,X)\bigr).
\]
Hence the exact unification type and the exact algebraic unification type of $\V$ coincide. 
\end{theorem}

\begin{proof}
Consider $\si\colon\Fml(X)\to \Fml(Y)$ in $\ExUn_{\V}(\Si,X)$.
Let $\hat{\si}\colon\FreeV(X)\rightarrow h_{\V}(\si(\Fml(X)))$ be the unique homomorphism determined by its value on the variables as follows:
\[
\hat{\si}(h_{\V}(x))=
h_{\V}(\si(x)) \mbox{ for each }x\in X.
\]
Then $\Si\subseteq\ker(\hat{\si}\circ h_{\V})$, and there exists a homomorphism $u_{\si}\colon \FpV(\Si,X) \to h_{\V}(\FreeV(Y))$ such that 
\begin{equation}\label{Eq:SyntacticAlgebraicSquare}
u_{\si}\circ \rho_{\Si,X,\V} = h_{\V}\circ\si.
\end{equation} 
Therefore, the map $u_{\si}$ is onto $h_{\V}(\si(\Fml(X)))$. Since $h_{\V}(\si(\Fml(X)))$ is a finitely generated subalgebra of $\FreeV(Y)$, \ $u_{\si}\in \EUAlg_{\V}(\FpV(\Si,X))$.

Let $u \colon\FpV(\Si,X)\to \E$ be a coexact-unifier for $\FpV(\Si,X)$. 
Since $\E$ is exact, there exist some finite set $Y$ and a one-to-one homomorphism 
$\iota\colon \E\to \FreeV(Y)$. 
For each $x\in X$, let $t_x\in\Fml(Y)$ such that 
$h_{\V}(t_x)=\iota(u(\rho_{\Si,X,\V}(x)))$.
Let $\si\colon \Fml(X)\to \Fml(Y)$ be the substitution defined
by $\si(x)=t_x$ for each $x\in X$. 
It is straightforward to check that 
$\iota\circ u=u_{\si}$ and $\iota(\E)= u_{\si}(\FpV(\Si,X))$. Since $\iota$ is one-to-one, there exists a homomorphism $\eta\colon u_{\si}(\FpV(\Si,X))\to \E$  that is the inverse of $\iota$. Therefore $u$ and $u_{\si}$ are equivalent in the preorder $\EUAlg_{\V}(\FpV(\Si,X))$.

By \eqref{Eq:SyntacticAlgebraicSquare}, for each $\si_1,\si_2\in \ExUn_{\V}(\Si,X)$
\begin{eqnarray*}
\si_2\eleq_{\V} \si_1&\Leftrightarrow& \ker(h_{\V}\circ\si_1)\subseteq\ker(h_{\V}\circ\si_2)\\
&\Leftrightarrow&
\ker(u_{\si_1}\circ \rho_{\Si,X,\V})\subseteq\ker(\si_2\circ \rho_{\Si,X,\V})\\
&\Leftrightarrow&\ker(u_{\si_1})\subseteq \ker(u_{\si_2}).
\end{eqnarray*}
Let us denote the codomains of $u_{\si_1}$ and $u_{\si_2}$ by $\E_1$ and $\E_2$, respectively.
Since $u_{\si_1}$ is onto~$\E_1$, \ $\ker(u_{\si_1})\subseteq \ker(u_{\si_2})$ iff there exists 
$h\colon \E_1\to \E_2$ such that $h\circ u_{\si_1}=u_{\si_2}$, that is $u_{\si_2}\leq u_{\si_1}$.

We have proved that the assignment $\si \mapsto u_{\si}$ determines an equivalence between the preorders $\ExUn_{\V}(\Si,X)$ and $\EUAlg_{\V}(\FpV(\Si,X))$. 
Hence, the result follow by Lemma~\ref{Lemma:EquivPreorder}.
\end{proof}

In the remainder of this section we present some consequences of the algebraic description of exact unification. 
Given an algebra  $\A$ in  $\V$, let $\cg_e(\A)$ denote the set of congruences $\theta$ of $\A$ such that the quotient $\A/\theta$ is exact; i.e., 
\[
\cg_e(\A)=\{\theta\in\cg(\A)\mid \A/\theta\in\cop{IS}(\FreeV(\omega))\}.
\]

\begin{theorem}
For any $\A \in \mathsf{FP}(\V)$:
\begin{itemize}
\item[(i)] given any homomorphism $u\colon A\to B$, 			
\[
			(u,u(\A)) \in \EUAlg_{\V}(\A) \quad \Leftrightarrow \quad \ker(u)\in \cg_e(\A).
			\]
\item[(ii)] $(u,\B),(v,\alg{C})\in \EUAlg_{\V}(\A)$ are such that $u\leq v$ iff $\ker(v)\subseteq \ker(u)$.
\end{itemize}
Hence $\ker\colon \EUAlg_{\V}(\A)\to \cg(\A)$ determines an equivalence between the preordered set $\EUAlg_{\V}(\A)$ 
and the poset $(\cg_e(\A),\supseteq)$.
\end{theorem}
\begin{proof}
(i) $(u,u(\A))\in \EUAlg_{\V}(\A)$ iff  $u(\A)\in\cop{IS}(\FreeV(\omega))$ iff $\ker(u)\in \cg_e(\A)$.

(ii) $u\leq v$ iff there exists a homomorphism $f\colon \alg{C}\to \B$ such that $f\circ v=u$ iff 
(as $v$ is surjective) $\ker(v)\subseteq \ker(u)$.
\end{proof}

\begin{corollary}\label{cor:Etype-Cong}
For each finitely presented algebra $\A$  in $\V$, 
\[
\textstyle\tp(\EUAlg_{\V}(\A))=\begin{cases}
1&\mbox{ if }|\min(\cg_e(\A))|=1;\\
\omega& \mbox{ if }1<|\min(\cg_e(\A))|<\infty;\\
\infty& \mbox{ if }\infty\leq|\min(\cg_e(\A))| ;\\
0& \mbox{ if } \min(\cg_e(\A))=\emptyset.
\end{cases} 
\] 
\end{corollary}

\begin{corollary}\label{Cor:LocFin}
Let $\V$ be a locally finite equational class. Then $\tp(\EUAlg_{\V}(\A))$ is  finite 
for  each $\A \in \mathsf{FP}(\V)$.  
Hence $\V$ has unitary or finitary exact unification type.
\end{corollary}
\begin{proof}
As $\V$ is locally finite, each finitely generated algebra in $\V$ is finite. In particular $\A$ is finite. Since 
$|\min(\cg_e(\A))|\leq |\mathcal{P}(A\times A)|=2^{|A|^2}$, where $\mathcal{P}(A\times A)$ denotes the 
powerset of $A\times A$, by Corollary~\ref{cor:Etype-Cong}, $\tp(\EUAlg_{\V}(\A))$ is either unitary or finitary.
\end{proof}

\begin{corollary}\label{Cor:TypeUnitary}
Let $\A$ be a finitely presented algebra in $\V$ such that its congruences are totally ordered. 
If $\EUAlg_{\V}(\A) \neq \emptyset$, then it is totally ordered and $\tp(\EUAlg_{\V}(\A))\in\{1,0\}$. 
In particular, if $\A$  is simple, then either  $\EUAlg_{\V}(\A)$ is empty or $\tp(\EUAlg_{\V}(\A))=1$.
\end{corollary}


\section{Case Studies} \label{s:case_studies}

\begin{table}[t]
\begin{center}
\begin{tabular}{|@{\ \ }c@{\ \ }|@{\ \ }c@{\ \ }|@{\ \ }c|}
\hline
Equational Class			& Unification Type 	& 	Exact Type\\
\hline\hline&&\\[-.35cm]
Boolean Algebras				&	Unitary		&	Unitary\\
\hline&&\\[-.35cm]
Heyting Algebras				&	Finitary		&	Finitary\\
\hline&&\\[-.35cm]
Semigroups					&	Infinitary		&	Infinitary or Nullary\\
\hline&&\\[-.35cm]
Modal algebras 				&	Nullary		&	Nullary\\
\hline 
&&\\[-.35cm]
Distributive Lattices			&	Nullary		&	Unitary\\
\hline&&\\[-.35cm]
Stone Algebras				&	Nullary		&	Unitary\\		
\hline&&\\[-.35cm]
Bounded Distributive Lattices	&	Nullary		&	Finitary\\
\hline&&\\[-.35cm]
Pseudocomplemented Distributive Lattices	& Nullary	& Finitary\\
\hline&&\\[-.35cm]
Idempotent Semigroups		&	Nullary		&	Finitary\\
\hline&&\\[-.35cm]
De Morgan Algebras		&	Nullary		&	Finitary\\
\hline&&\\[-.35cm]
Kleene Algebras		&	Nullary		&	Finitary\\
\hline&&\\[-.35cm]
MV-algebras					&	Nullary		&	Finitary\\
\hline&&\\[-.35cm]
Willard's Example					&	Infinitary		&	Finitary\\
\hline
\end{tabular}
\end{center}
\caption{Comparison of unification types and exact types}
\label{table}
\end{table}

Any unitary equational class such as the class of Boolean algebras also has exact 
unitary type, and any finitary equational class will have unitary or finitary exact type. 
For example, the class of Heyting algebras is finitary~\cite{Ghi99} and hence 
also has finitary exact type (the equation $x \lor y \eq \top$ has two most general exact unifiers 
as in Example~\ref{e:equnif}). 
Minor changes to the original proofs that the class of semigroups has infinitary unification type~\cite{Plo72} 
and that the class of modal algebras (for the logic $\lgc{K}$) has nullary unification type~\cite{Jer13} 
establish that the former has infinitary or nullary exact type and the latter has nullary exact type. 
Below we consider more interesting cases where the type changes, collecting 
the results in Table~\ref{table}.

\begin{example}[Distributive Lattices]
However, the class of 
distributive lattices, which is known to have nullary unification type~\cite{Ghi97}, has unitary exact 
type as all finitely presented distributive lattices are exact (see for example~\cite[Lemma~18]{CM13}). The classes of bounded distributive lattices \cite{Ghi97}, idempotent semigroups (or bands)  
 \cite{Ba86}, De Morgan, and 
Kleene algebras~\cite{BC201X}  are also nullary, but because all these classes are locally finite, they have at most 
-- and indeed, it can be shown via suitable cases, precisely -- finitary exact type.
\end{example}

\begin{example}[Pseudocomplemented Distributive Lattices]
The equational class $\mathfrak{B}_{\omega}$ of pseudocomplemented distributive 
lattices is the class of algebras $(B, \land, \lor, ^*, 0, 1)$ such that $(B, \land, \lor, 0, 1)$ is a 
bounded distributive lattice and $a \land b^* = a$ if and only if $a \land b = 0$ for all $a,b \in B$. 
For each $n\in \mathbb{N}$, let $\B_n=(B_n,\wedge,\vee,^{*},0,1)$ denote the finite Boolean algebra 
with $n$ atoms and let $\B_n'$ be the algebra obtained by adding a new top~$1'$ to the underlying lattice 
of $\B_n$ and endowing it with the unique operation making it into a pseudocomplemented 
distributive lattice. Let $\mathfrak{B}_{n}$ denote the subvariety of $\mathfrak{B}_{\omega}$ generated by $\B_n'$.
Lee proved in \cite{Lee70}, that the subvariety lattice of $\mathfrak{B}_{\omega}$ is 
\[
\mathfrak{B}_{0}\subsetneq\mathfrak{B}_{1}\subsetneq\cdots \subsetneq \mathfrak{B}_{n}\subsetneq\cdots \subsetneq \mathfrak{B}_{\omega}
\]
where $\mathfrak{B}_{0}$ and $\mathfrak{B}_{1}$ are the varieties of Boolean algebras and  Stone algebras, respectively.
We have already observed that the class of Boolean algebras has exact type $1$.
The case of Stone algebras is similar to  distributive lattices: $\mathfrak{B}_1$ 
has nullary unification~\cite{Ghi97} type; however, all finitely presented Stone algebras are exact 
(see~\cite[Lemma~20]{CM13}), so the class of Stone algebras has unitary exact type.

In \cite{Ghi97} it was proved that $\mathfrak{B}_{\omega}$ has nullary unification type, 
and the same result was proved in \cite{CaXX} for $\mathfrak{B}_n$ for each $n\geq2$. 
All these varieties are locally finite, so an application of Corollary~\ref{Cor:LocFin} proves that they 
have at most finitary unification type. 
It is easy to prove that $x\vee\neg x\eq \top \imp x\eq \top,\neg x\eq\bot$ is admissible in $\mathfrak{B}_{\omega}$  and $\mathfrak{B}_{n}$ for each $n\geq 2$ and that neither
$x\vee\neg x\eq \top \imp x\eq \top$  nor 
$x\vee\neg x\eq \top \imp \neg x\eq\bot$  are admissible in $\mathfrak{B}_{\omega}$ or $\mathfrak{B}_{n}$ with $n\geq 2$.
By Corollary~\ref{Cor:TypeUnitary}, the classes $\mathfrak{B}_{\omega}$ and $\mathfrak{B}_{n}$ with $n\geq 2$ have finitary type. 
\end{example}

\begin{example}[A Locally Finite Equational Class with Infinitary Unification Type] 
The following example of a locally finite equational class with infinitary unification type 
is due to R.~Willard (private communication). Consider a language with one binary 
operation, written as juxtaposition, and two constants $0$ and $1$. 
Let $\V$ be the  equational class  defined by 
\[
0x \eq x0 \eq 0, \quad 1 x \eq 0, \quad x(yz) \eq 0, \quad (x1)1 \eq x1,
\]
and, for each $n \in \mathbb{N}$, associating to the left,
\[
xyz_1z_2\ldots z_n y \eq xyz_1z_2\ldots z_n 1.
\]
Then up to equivalence, terms have the form (again associating to the left)
\[
0, \quad 1, \quad \mbox{or} \quad xy_1y_2\ldots y_n
\]
where $y_1,\ldots,y_n$  are variables or $1$ and all distinct, and $x$ is any variable. 
It is immediate that finitely generated free algebras are finite and hence that $\V$ is 
locally finite. Note also that $\{xy \eq 0\}$ has three most general exact unifiers 
\[
\si_1(x) = 1,\ \si_1(y) = y; \quad \si_2(x) = 0,\ \si_2(y) = y;\quad \si_3(x) = x,\ \si_3(y) = yz.
\]
So the exact unification type of $\V$ is finitary.

We now claim that the following set of identities has infinitary unification type:
\[
\Si = \{xy \eq x1\}.
\]
For each $n \in \mathbb{N}$ and distinct variables $z_1,\ldots,z_n$ different from $y$, consider the following $\V$-unifier of $\Si$:
\[
\si_n(x) = xyz_1 \ldots z_n, \quad \si_n(y) = y.
\]
Then the set $\{\si_n \mid n \in \mathbb{N}\}$ is a $\mu$-set for $\U{\V}{\Si}$. Moreover, it can be shown 
that no set of identities has nullary unification type.
\end{example}

\begin{example}[MV-algebras]
In \cite{MS13} it is proved that the equational class $\mathcal{MV}$ of MV-algebras has nullary unification type. 
This class is not locally finite, so we cannot apply Corollary~\ref{Cor:LocFin}. However, combining results 
from~\cite{Jer09b} and~\cite{Cab}, we can still prove that MV-algebras have finitary exact type.


Let $\lang$ be the language of MV-algebras and $\Si$ a finite set of equations in $\FML$. 
Finitely presented MV-algebras admit a presentation of the form 
$\{\a\eq \top\}$, so there is no loss of generality in assuming that $\Si=\{\a\eq \top\}$. Let us fix $X=\var(\a)$ 
and $\A=\Fp_{\mathcal{MV}}(\{\a\eq\top\})$.
A combination of  \cite[Theorem~3.8]{Jer09b} and \cite[Theorem~4.18]{Cab} proves the following result.
There exist $\b_1,\ldots,\b_n\in\Fml(X)$ such that the following hold:
\begin{itemize}
\item[(i)] the rule $\{\a\eq\top\}\imp\{\b_1\eq\top,\ldots,\b_n\eq\top\}$ is admissible  in $\mathcal{MV}$;
\item[(ii)] $\{\b_i\eq\top\} \mdl{\mathcal{MV}} \a\eq\top$ for each $i\in\{1,\ldots,n\}$;
\item[(iii)] $\Fp_{\mathcal{MV}}(\{\b_i\eq\top\})$ is exact for each $i\in\{1,\ldots,n\}$.
\end{itemize}
Defining $\B_i=\Fp_{\mathcal{MV}}(\{\b_i\eq\top\})$, from (ii), we obtain that for each $i\in\{1,\ldots,n\}$, there exists a homomorphism $e_i\colon\A\to \B_i$ such that $\rho_{\{\b_i\eq\top\},X,\mathcal{MV}}=e_i\circ\rho_{\{\a\eq\top\},X,\mathcal{MV}}$. Since $\rho_{\{\b_i\eq\top\},X,\mathcal{MV}}$ is onto, so is $e_i$. 
By (iii), it follows that $S=\{e_1,\ldots,e_n\}$ is a set of coexact unifiers of $\A$. We claim that $S$ is a complete set in $\EUAlg_{\mathcal{MV}}(\A)$. 
Indeed, let $e\colon \A\to\alg{C}\in\EUAlg_{\mathcal{MV}}(\A)$. By (i), there exists $i\in\{1,\ldots,n\}$ and $h\colon \B_i\to \alg{C}$  such that $e\circ\rho_{\{\a\eq\top\},X,\mathcal{MV}}=h\circ\rho_{\{\b_i\eq\top\},X,\mathcal{MV}}$. 
Since $\rho_{\{\a\eq\top\},X,\mathcal{MV}}$ is onto and $\rho_{\{\b_i\eq\top\},X,\mathcal{MV}}=e_i\circ\rho_{\{\a\eq\top\},X,\mathcal{MV}}$, it follows that $e=h\circ e_i$, that is, $e\leq e_i$. 
This proves that $\tp(\EUAlg_{\mathcal{MV}}(\A))\in\{1,\omega\}$, hence  the exact type of $\mathcal{MV}$ is either unitary or finitary. 
By \cite[Lemma~4.2]{Jer09b}, $x\vee\neg x\eq \top \imp x\eq \top,\neg x\eq\bot$ is admissible in $\mathcal{MV}$ and it is easy to see that neither $x\vee\neg x\eq \top \imp x\eq \top$  nor $x\vee\neg x\eq \top \imp \neg x\eq\bot$  are admissible.
So by Corollary~\ref{Cor:TypeUnitary}, $\mathcal{MV}$ has finitary exact type. It is possible to 
prove that $\tp(\EUAlg_{\mathcal{MV}}(\Fp_{\mathcal{MV}}(\{x\vee\neg x\eq\top\})))=2$, but such a calculation is 
beyond the scope of this paper.
\end{example}


\section{Concluding Remarks} \label{s:concluding_remarks}

We have introduced a new hierarchy of exact unification types based on an inclusion preordering of 
unifiers, showing that in certain cases, the exact type reduces from nullary or infinitary unification type 
to finitary or even unitary exact type. Note, however, that we do not know if there are examples of 
equational classes of (i) finitary unification type that have unitary exact type, (ii) infinitary unification type 
that have unitary or nullary exact type, (iii) nullary unification type that have infinitary exact type.

In \cite{CM13}, the current authors present axiomatizations for 
admissible rules of several locally finite (and hence of finitary exact unification type) equational classes 
with classical unification type $0$. In all these cases a complete description of exact algebras, and the 
finite exact unification type plays a central (if implicit) role. We therefore 
expect that this approach will be useful for tackling other classes of algebras that have unitary or finitary exact type, 
independently of their unification type. 


\end{document}